\documentclass[12pt,twoside]{amsart}
\usepackage{amssymb}
\usepackage{graphicx}
\usepackage[margin=1in]{geometry}

\newcommand{\beas}{\begin{eqnarray*}}
\newcommand{\eeas}{\end{eqnarray*}}
\newcommand{\bea}{\begin{eqnarray}}
\newcommand{\eea}{\end{eqnarray}}
\newcommand{\beq}{\begin{equation}}
\newcommand{\eeq}{\end{equation}}
\newcommand{\ben}{\begin{enumerate}}
\newcommand{\een}{\end{enumerate}}

\newtheorem{theorem}{Theorem}[section]
\newtheorem{lemma}[theorem]{Lemma}
\newtheorem{proposition}[theorem]{Proposition}
\newtheorem{corollary}[theorem]{Corollary}
\newtheorem{conjecture}[theorem]{Conjecture}

\theoremstyle{definition}
\usepackage{latexsym}
\usepackage{color,hyperref}
\definecolor{darkblue}{rgb}{0,0,0.6}
\hypersetup{colorlinks,breaklinks,linkcolor=darkblue,urlcolor=darkblue,anchorcolor=darkblue,citecolor=darkblue}
\author[Sung Gi Park]{Sung Gi Park}
\address{Department of Mathematics, Harvard University, Cambridge, MA 02138}
\email{sgpark@math.harvard.edu}

\author[Richard P. Stanley]{Richard P. Stanley}
\address{Department of Mathematics, MIT, Cambridge, MA 02139}
\email{rstan@math.mit.edu}

\author[Fabrizio Zanello]{Fabrizio Zanello}
\address{Department of Mathematical Sciences, Michigan Tech, Houghton, MI 49931}
\email{zanello@mtu.edu}

\title[Proof of the Gorenstein Interval Conjecture in low socle degree]{Proof of the Gorenstein Interval Conjecture\\in low socle degree}
\begin{document}

\begin{abstract} 
Roughly ten years ago, the following ``Gorenstein Interval Conjecture'' (GIC) was proposed:  Whenever   $(1,h_1,\dots,h_i,\dots,h_{e-i},\dots,h_{e-1},1)$ and $(1,h_1,\dots,h_i+\alpha ,\dots,h_{e-i}+\alpha,\dots,h_{e-1},1)$ are both Gorenstein Hilbert functions for some $\alpha \geq 2$, then $(1,h_1,\dots,h_i+\beta ,\dots, h_{e-i}+\beta,\dots,h_{e-1},1)$ is also Gorenstein, for all $\beta =1,2,\dots,\alpha -1$. Since an explicit characterization of which Hilbert functions are Gorenstein is widely believed to be hopeless, the GIC, if true, would at least provide the existence of a strong, and very natural, structural property for such basic functions in commutative algebra. Before now, very little progress was made on the GIC. 

The main goal of this note is to prove the case $e\le 5$, in arbitrary codimension. Our arguments will be in part constructive, and will combine several different tools of commutative algebra and classical algebraic geometry.
\end{abstract}


\maketitle

\section{Introduction} 

We consider \emph{standard graded artinian algebras} $A=R/I$, where $R = k[x_1,\dots,x_r]$ is a polynomial ring over a field $k$ of characteristic zero, $I$ is a homogeneous ideal of $R$, and the variables $x_i$ have degree one. Since we do not lose  generality by assuming  that $I$ does not contain nonzero linear forms,  $r$ is the \emph{codimension} of $A$.

Recall that the \emph{Hilbert function} (HF) of $A$ is defined as $h_i=\dim_k A_{i}$, for all $i\ge 0$. It is a standard fact in commutative algebra that $A$ is artinian if and only if the HF of $A$ is eventually zero. Thus, we can identify the HF of $A$ with its $h$-vector $h=(1,h_1,\dots,h_e)$, where $e$ is the largest index such that $h_e>0$ and is called the \emph{socle degree} of $A$. The \emph{socle}  of $A$ is the annihilator of the maximal ideal $(\overline{x_1}, \dots ,\overline{x_r})\subset A$. Finally, $A$ is \emph{level of type $t$} if its socle is a $t$-dimensional $k$-vector space all concentrated in degree $e$. We say that $A$ is \emph{Gorenstein} if $A$ is level of type 1.

An important theme in combinatorial commutative algebra is the understanding of level and Gorenstein HFs. In our paper, we present a new contribution in this direction.

The theory of level algebras began with a couple of seminal papers of the second author in the Seventies \cite{St1,St}, and has since been an active area of research, due to the intrinsic interest of the topic and also its applications to several disciplines as diverse as combinatorics, algebraic geometry, representation theory, and even computational complexity. We refer the interested reader to \cite{AMS,AS,BI,BL,BZ,CI,GHMS,Ia1,IK,IS,MNZ1,MNZ3,MNZ2,MZ,SS,We,Za1,Za2} as a highly nonexhaustive list of recent works and further sources.

Despite the considerable progress of the last several years, however, it is widely accepted by experts that an explicit classification of Gorenstein and level HFs is in general hopeless. The only two broad classes that have been characterized to date are those of level HFs  of codimension two (a simple result already known to Macaulay \cite{Ma}; see also \cite{Ia1}), and Gorenstein HFs of codimension three (\cite{St}; see \cite{Za} for a combinatorial proof). 

Absent a full characterization, it is an interesting problem to determine strong structural results for Gorenstein and level HFs. One such general, and very natural, conjectural statement was suggested roughly ten years ago by the third author \cite{Za3}. We recall here the statement in the Gorenstein case, and refer to \cite{Za3} for the conjecture in higher type.

\begin{conjecture}\label{GIC} \emph{(Gorenstein Interval Conjecture (GIC))}. Suppose that for some $\alpha \geq 2$,  $(1,h_1,\dots,h_i,\dots,h_{e-i},\dots,h_{e-1},1)$ and $(1,h_1,\dots,h_i+\alpha ,\dots,h_{e-i}+\alpha,\dots,h_{e-1},1)$ are Gorenstein HFs. Then $$(1,h_1,\dots,h_i+\beta ,\dots, h_{e-i}+\beta,\dots,h_{e-1},1)$$ is also Gorenstein, for each $\beta =1,2,\dots,\alpha -1$.
\end{conjecture}

While  the GIC would appear consistent with all known tools commonly employed in this area (algebraic, combinatorial, homological, geometric), precious little progress has so far been made towards its solution. In this note, we employ a novel combination of techniques coming from both commutative algebra and algebraic geometry, and prove the GIC in socle degree $e\le 5$, in any codimension. We wish to point out here that geometric tools of a similar nature, which allow one to focus on restrictions modulo general linear forms, were previously employed in \cite{MNZ3}, though the arguments and the results of this paper are very different. Also interesting, we reduce our main theorem to the (equivalent) fact that suitable functions $f_e$ of the codimension, first defined by the second author in his studies of the nonunimodality of Gorenstein HFs \cite{St,StCCA}, are nondecreasing.

\section{Some technical tools and preliminary notions}

We first briefly recall a useful tool in the study of  HFs, namely \emph{Macaulay's inverse systems}. We refer to the standard sources \cite{ger,IK} for more information and a comprehensive treatment.  For a a  homogeneous ideal $I\subset R=k[x_1,\dots,x_r]$, define its \emph{inverse system} to be the graded $R$-module $I^{\perp}=M\subset
S=k[y_1,\dots,y_r]$, where $R$ acts on $S$ by partial differentiation and $I=\operatorname{ann}(M)$. The external product defining the module structure on $S$ is thus uniquely determined by linearity by $x_i \circ F= \partial_iF,$ for any form $F\in S$ and $i=1,\dots,r$. 

A crucial consequence for us here is that the HFs of $A=R/I$ and $M$ coincide; i.e., $\dim_k A_i = \dim_k M_i$, for all $i$. Moreover, $A$ is artinian, level of type $t$ and socle degree $e$ if and only if $M$ is finitely generated by $t$ linearly independent forms of degree $e$.% For simplicity, we may sometimes consider the form $F$ to also lie inside of $R$.

In particular, any Gorenstein HF of codimension $r$ and socle degree $e$ can be obtained by computing the dimensions of the spaces of partial derivatives of a form $F\in k[y_1,\dots,y_r]$ of degree $e$ in \emph{essentially} $r$ variables (i.e., the first partials of $F$ span a vector space of dimension precisely $r$). Because Gorenstein HFs are symmetric, note that the case $e=3$ is completely understood (all such HFs are of the form $(1,r,r,1)$, for any $r\ge 1$). For $e=4$ and 5, Gorenstein HF are of the form $(1,r,a,r,1)$ and $(1,r,a,a,r,1)$, respectively. 

A considerable amount of work has been devoted to understanding the possible values of $a$ as a function of $r$. If we define $f_e(r)$ as the least possible degree 2 entry of a Gorenstein HF of socle degree $e$ and codimension $r$, then $f_e(r)$  has been determined asymptotically for all $e$; in particular, $f_4(r)\sim_r(6r)^{2/3}$ (as conjectured in \cite{St1} and proven in \cite{MNZ1}) and $f_5(r)\sim_r \frac{1}{6}(24r)^{3/4}$ (see \cite{MNZ2}, Theorem 3.6 for more). Also, it was  recently shown \cite{MZ} that the first nonunimodal example of a Gorenstein HF ever produced, namely $(1,13,12,13,1)$ \cite{St1}, is the smallest possible in terms of dimension of the algebra. In particular, $f_4(13)=12$ and $f_4(r)<r$ if and only if $r\ge 13$. Similarly, $f_5(r)<r$ precisely for $r\ge 17$ (see \cite{AMS,MZ} for more in this direction). For most $r$, however, the precise values of $f_4$ and $f_5$ remain unknown. 
%Two very basic (and highly nontrivial) questions on Gorenstein HFs of socle degree 4 and 5 that were still open after the above works are: 1) Do these Gorenstein HFs satisfy the GIC? 2) Are the above functions $f_4$ and $f_5$ nondecreasing with $r$? In this note, as we mentioned earlier, by means of a combination of techniques coming from combinatorial commutative algebra and classical algebraic geometry, we will see that the two questions are in fact equivalent, and answer them positively.

To prove our main result, we will rely on the (simpler) fact that the GIC is already known to hold in degree 2 when $e\le 5$. More precisely, we have:

\begin{lemma}\label{gic2} \emph{(\cite{Za3}, Theorem 2.3)}.
Let $(1,r,a,\dots)$ be a Gorenstein HF of socle degree $e\in \{4, 5\}$. Then $f_e(r)\le a\le \binom{r+1}{2}$, and any such integer value of $a$ can actually occur.
\end{lemma}

We will also need the following two lemmas. We refer the reader to Hartshorne \cite{Hartshorne} or Vakil \cite{vakil} for standard facts and unexplained terminology of algebraic geometry.

\begin{lemma}\label{dim} \emph{(\cite{Hartshorne}, Chapter 2, Exercise 3.22; or \cite{vakil}, Proposition 11.4.1)}.
Given a morphism of irreducible $k$-varieties $\pi:X\rightarrow Y$ with $\dim X=m$ and $\dim Y=n$, there exists a nonempty, Zariski-open subset $U\subset Y$ such that for all $q\in U$, $\pi^{-1}(q)$ has dimension $m-n$ or is empty. In particular, $\dim \pi(X) \le \dim X$.
\end{lemma}

\begin{lemma}\label{bertini} \emph{(\cite{Hartshorne}, Chapter 3, Corollary 10.9 and Remark 10.9.2; or \cite{vakil}, Exercise 25.3.D)}.
Let $f_0,\dots,f_n$ be forms of the same degree in $k[y_0,\dots,y_n]$, where $k$ is  algebraically closed of characteristic zero. Then, for any general linear combination $b_0f_0+\dots+b_nf_n$, the hypersurface $V(b_0f_0+\dots+b_nf_n)$ is smooth in $\mathbb P^n_k\setminus V(f_0,\dots,f_n)$.
\end{lemma}

\section {Proof of the GIC in socle degree $e\le 5$}

We begin by proving that the GIC holds in socle degree $e\le 5$ if and only if the corresponding functions $f_e(r)$ are nondecreasing. We will then show the latter fact, by a suitable application of Macaulay's inverse systems and of a (more general) geometric result.

\begin{proposition}\label{f}
Let $e=\{4, 5\}$. Then the GIC holds in socle degree $e$ if and only if $f_e$ is a nondecreasing function of $r$.
\end{proposition}

\begin{proof}
We show the result for $e=4$, the case $e=5$ being entirely similar. Assume $f_4$ is nondecreasing. By Lemma \ref{gic2}, it suffices to prove the GIC on the first entry of a Gorenstein HF; i.e., we want to show that if $(1,r_1,a,r_1,1)$ and $(1,r_2,a,r_2,1)$ are both Gorenstein and $r_1< r_2$, then  $(1,r,a,r,1)$ is also Gorenstein, for any $r_1<r<r_2$. 

Suppose that this is not the case; hence, $(1,r,a,r,1)$ is not Gorenstein for some $r_1<r<r_2$. Then  $f_4(r)>a$; otherwise, because of Lemma \ref{gic2}, $a>\binom{r+1}{2}> \binom{r_1+1}{2}$, which is impossible since $(1,r_1,a,r_1,1)$  is a Gorenstein HF. On the other hand, since $(1,r_2,a,r_2,1)$ is Gorenstein and $f_4$ is nondecreasing, we have that $f_4(r)\le f_4(r_2)\le a$, a contradiction.

Now suppose the GIC holds, and that for some $r_1<r$, $f_4(r_1)>f_4(r)$. Note that $f_4(r_1)\le r_1$, since $(1,r_1,r_1,r_1,1)$ is a Gorenstein HF (for instance, it is easily obtained from the inverse system form $y_1^4+y_2^4+\dots+y_{r_1}^4\in k[y_1,\dots,y_{r_1}]$). But, similarly, $\left(1,f_4(r),f_4(r),f_4(r),1\right)$ is also Gorenstein. Therefore, the inequalities $f_4(r)<r_1<r$ force the failure of the GIC, since our assumptions imply that $\left(1,r_1,f_4(r),r_1,1\right)$ cannot be Gorenstein. This contradiction gives us that $f_4$ is nondecreasing.
\end{proof}

\begin{lemma}\label{divisibility}
Let $n\ge 2$, and suppose $f_0,\dots, f_n\in k[y_0,\dots,y_n]$ are linearly independent forms of the same degree $d>1$ with $\gcd(f_0,\dots, f_n)=1$. Then, for any general linear form $H=\alpha_0y_0+\dots+\alpha_ny_n$, no nonzero linear combination of the $f_i$ over $k$ is  divisible by $H$.% That is,
%$$\alpha_0y_0+\dots+\alpha_ny_n {\ }\nmid {\ } b_0f_0+\dots+b_nf_n, {\ } {\ }\text{for all}{\ }{\ } [b_0:b_1:\dots:b_n]\in \mathbb P_k^n.$$
\end{lemma}

\begin{proof}
Let $\mathbb P_k^n=\left\{[b_0:\dots:b_n]\right\}$ and ${\mathbb P_k^n}^\vee =\left\{[a_0:\dots:a_n]\right\}$ be the projective spaces parameterizing, respectively, the nonzero linear combinations $b_0f_0+\dots+b_nf_n$ and the nonzero linear forms $a_0y_0+\dots+a_ny_n$. Define $Z$ as the following subvariety of ${\mathbb P_k^n}^\vee \times \mathbb P_k^n$: %defined by the determinant of the resultant matrix induced by the relation $a_0y_0+\dots+a_ny_n \mid b_0f_0+\dots+b_nf_n$. It is clear that $Z$ parametrizes hyperplane sections dividing linear combinations of the $f_i$; that is:
$$Z=\left\{([a_0:\dots:a_n],[b_0:\dots:b_n]) {\ }:{\ } a_0y_0+\dots+a_ny_n{\ }|{\ }b_0f_0+\dots+b_nf_n \right\},$$
and consider  the projections $\pi_1$ and $\pi_2$ of $Z$ to ${\mathbb P^n_k}^{\vee}$ and $\mathbb P_k^n$, respectively. Note that it suffices to show that $\pi_1$ is not surjective, since then the lemma is proven by choosing any linear form $H=\alpha_0y_0+\dots+\alpha_ny_n$ with coefficients in  ${\mathbb P^n_k}^\vee\setminus \pi_1(Z)$.

We argue by contradiction, and suppose $\pi_1$ is surjective. Since surjectivity is preserved by extending scalars, we may assume $k$ is algebraically closed. By Lemma \ref{bertini}, there exists a nonempty open subset $U\subset \mathbb P^n_k$ such that for every $[b_0:\dots:b_n]\in U$, $V(b_0f_0+\dots+b_nf_n)$ is smooth on the complement,  in the ambient projective space, of the base locus $V(f_0,\dots,f_n)$.

Fix a point $[b_0:\dots:b_n]\in U$, and suppose there exists some $[a_0:\dots:a_n]$ such that $b_0f_0+\dots+b_nf_n$ factors as $(a_0y_0+\dots+a_ny_n)\cdot f$. Then $V(b_0f_0+\dots+b_nf_n)$ is singular on $V(a_0y_0+\dots+a_ny_n,f)$. Since, as we saw above, $V(b_0f_0+\dots+b_nf_n)$ is smooth on the complement of  $V(f_0,\dots,f_n)$, we deduce that 
$$
V(a_0y_0+\dots+a_ny_n,f)\subset V(f_0,\dots,f_n),
$$
which is equivalent to the inclusion of radical ideals $\sqrt{(f_0,\dots, f_n)}\subset \sqrt{(a_0y_0+\dots+a_ny_n,f)}$. 

Since by Krull's principal ideal theorem, $(a_0y_0+\dots+a_ny_n,f)$ is an ideal of height at most $2$, $(f_0,\dots, f_n)$ is included in some prime ideal $P$ of height $2$ containing $a_0y_0+\dots+a_ny_n$. Noting that $\gcd(f_0,\dots, f_n)=1$ gives us that $(f_0,\dots, f_n)$ has height precisely 2, we define $T$ as the set of primes of height 2 containing $(f_0,\dots, f_n)$. Clearly, $T$ is finite because of the primary decomposition of $(f_0,\dots, f_n)$.

The above argument shows that whenever $[b_0:\dots:b_n]\in U$ and $a_0y_0+\dots+a_ny_n$ divides $b_0f_0+\dots+b_nf_n$, there exists $P\in T$ such that $a_0y_0+\dots +a_ny_n\in P$. Now let $S\subset {\mathbb P_k^n}^\vee$ consist of those points $[a_0:\dots:a_n]$ such that $a_0y_0+\dots+a_ny_n \mid b_0f_0+\dots+b_nf_n$ for some $[b_0:\dots:b_n]\in U$; i.e., $S=\pi_1(\pi_2^{-1}(U))$. Since a prime of height 2 contains at most a 2-dimensional subspace of linear forms, each element of $T$ contributes to a point or a line in ${\mathbb P_k^n}^\vee$. Thus, the finiteness of $T$ implies that $S$ is contained in the union of finitely many lines and points in ${\mathbb P^n_k}^\vee.$

On the other hand, only finitely many linear forms can divide a given linear combination of $f_0,\dots,f_n$, so the fiber of $\pi_2$ over any point of $\mathbb P^n_k$ is finite. Therefore, it easily follows from Lemma \ref{dim} that $\pi_2^{-1}\left(\mathbb P_k^n\setminus U\right)$ is a subvariety of $Z$ of dimension at most the dimension of $\mathbb P_k^n\setminus U$. We conclude, again by employing Lemma \ref{dim}, that
$$\dim \pi_1\left(\pi_2^{-1}(\mathbb P_k^n\setminus U)\right)\le \dim \pi_2^{-1}\left(\mathbb P_k^n\setminus U\right) \le n-1.$$

Because we have assumed $\pi_1$ is surjective, $\pi_1(\pi_2^{-1}(U))$ must contain a nonempty open subset of ${\mathbb P^n_k}^\vee$. Hence $S$ is dense in ${\mathbb P^n_k}^\vee$, a contradiction since a finite union of lines and points cannot cover a dense set in ${\mathbb P^n_k}^\vee$ for $n\ge2$.
\end{proof}

\begin{lemma}\label{gcd}
\label{gcd}
Suppose $F=p_1^{e_1}p_2^{e_2}\dots p_t^{e_t}$ is an irreducible factorization of a form $F\in k[y_0,\dots,y_n]$. Then
$$
\gcd(\partial_0F,\partial_1F,\dots, \partial_nF)=p_1^{e_1-1}\cdots p_t^{e_t-1}.
$$
\end{lemma}

\begin{proof}
Let $p$ be an irreducible polynomial dividing $\gcd(\partial_0F,\partial_1F,\dots, \partial_nF)$. Since $F=\sum_{i=0}^n y_i(\partial_i F)/\deg F$, it follows that $p$ divides $F$, and therefore it must coincide with some $p_i$ because of the uniqueness of the factorization. Now note that for any index $i$,
$$\partial_i F=\sum_{j=1}^t e_j(\partial_i p_j)\frac{F}{p_j}.$$

Thus, since $\partial_i p_j$ shares no common divisor with $p_j$, the order of $p_j$ in $\partial_i F$ is precisely $e_j-1$. This proves the lemma.
\end{proof}

Now consider a Gorenstein algebra $A$ of arbitrary socle degree $e\ge 3$ and codimension $r=n+1\ge 3$, whose corresponding inverse system $M$ is generated by some form  $F\in S=k[y_0,\dots,y_n]$ of degree $e$. For a Zariski-general linear form $H$, let $F^H$ be the image of $F$ in the polynomial ring in essentially $n$ variables $S^H=k[y_0,\dots,y_n]/H$.

If
$$\left(1,h_1(F),\dots, h_{e-1}(F),h_e(F)=1\right)$$
is the HF of $A$, and
$$\left(1,h_1\left(F^H\right),\dots, h_{e-1}\left(F^H\right),h_e\left(F^H\right)=1\right)$$
is the HF of the  Gorenstein algebra $A^H$ corresponding to the inverse systems module $M^H=\langle F^H\rangle \subset S^H$, then it is easy to see that
$$h_i\left(F^H\right)\le h_i(F)$$
for all $i$, and 
$$h_1\left(F^H\right)\le n.$$

We are now ready to show that, since $H$ is general, $h_1\left(F^H\right)$ is in fact equal to $n$.

\begin{theorem}\label{n}
With the above notation,
$$h_1\left(F^H\right)= n.$$
\end{theorem}

\begin{proof}  %J=M and \bar J^H = M^H
By inverse systems,  since the codimension of $A$ is $n+1$, we have $$\dim_k M_1=\dim_k M_{e-1}=n+1,$$
where $M_{e-1}=\langle \partial_0 F,\dots,\partial_n F\rangle$ is the $k$-vector space spanned by the first partials of $F$. Thus, we want to show that $\dim_k M^H_1=\dim_k M^H_{e-1}=n$, for a general linear form $H=\alpha_0y_0+\dots+\alpha_ny_n$. It is easy to see that the theorem is proven whenever we show that no nonzero element in $M_{e-1}$ (i.e., no nonzero linear combination of the $\partial_i F$) is  divisible by $H$.

Set $g=\gcd(\partial_0F,\dots, \partial_nF)$ and $f_i=\partial_iF/g$, for all $i$. When $g=1$, Lemma \ref{divisibility} proves the theorem. Thus we can assume $g$ has positive degree. The $f_i$ are clearly all forms of the same degree, say $d$. If $d>1$,  then we are again done by Lemma \ref{divisibility}, since only finitely many linear forms can divide $g$ and we are assuming that $H$ is general. 

Hence let $d\le 1$. If $p_1^{e_1}p_2^{e_2}\dots p_t^{e_t}$ is an irreducible factorization of $F$, then it easily follows from Lemma \ref{gcd} that $d=\left(\sum_{j=1}^t \deg p_j\right) -1$. Thus,  $\sum_{j=1}^t \deg p_j \le 2$. 

We deduce that $F$ can only be equal to $p_1^{m_1}p_2^{m_2}$, with $p_1$ and $p_2$ linear forms, or to $p^m$, with $p$ linear or quadratic. 
Because any quadratic form can be diagonalized, after a change of variables we promptly reduce $F$ to the following three possible cases:
$$
F=
\begin{cases} y_0^{m_1}y_1^{m_2}; \\ 
y_0^m;\\
\left(\sum_{i=0}^n c_iy_i^2\right)^m.\end{cases}
$$

In the first two cases, the first derivatives of  $F$ span a space of dimension at most 2, against the assumption $\dim_k M_{e-1}=n+1\ge 3$. Thus, $F=\left (\sum_i c_iy_i^2\right)^m$. Note that $c_i\neq 0$ for all $i$, since the first derivatives of $F$ are linearly independent. It is now immediate to see that, e.g., $H=y_0$ yields $\dim_k M^H_{e-1}=n$. Thus, the result for $H$ general follows by semicontinuity.
\end{proof}

\begin{corollary}
The GIC holds in socle degree $e\le 5$.
\end{corollary}

\begin{proof}
Assume $r\ge 3$, otherwise the result is trivial. By Proposition \ref{f}, it suffices to show that $f_4$ and $f_5$ are nondecreasing functions of $r$. Theorem \ref{n} implies that if $$(1,h_1=r,h_2=f_e(r),\dots,h_{e-1},h_e=1)$$ is a Gorenstein HF, then there exists a Gorenstein HF $$(1,h_1'=r',h_2',\dots,h_{e-1}',h_e'=1)$$ with $h_i'\le h_i$ for all $i$ and $r'=r-1$. This immediately gives $f_e(r-1)\le h_2'\le f_e(r)$, which completes the proof.
\end{proof}

\section{Acknowledgements} We are grateful to the referee for their careful reading of our manuscript and for pointing out some minor issues in the original version. Most of the research contained in this paper took place during a visiting professorship of FZ in Fall 2017, for which he warmly thanks RPS and the MIT Math Department. He also thanks Tony Iarrobino and Juan Migliore for comments. The work of SGP was done as part of a research project supervised by FZ, for which SGP is grateful to MIT for financial support via a UROP grant. FZ was partially supported by a Simons Foundation grant (\#274577). Finally, RPS and FZ would like to acknowledge that the contribution of SGP went way beyond the expected role of an undergraduate student, and was in fact crucial to the success of this project.

%%%%%%%%%%%%%%%%%%%%%%%%%%%%%%%%%%%%%%%%%%%%%%%%%%%%%%%%%%%%%%

\end{document}